\documentclass[11pt]{amsart}

\usepackage[margin=1in]{geometry}  
\usepackage{graphicx}              
\usepackage{amsmath}              
\usepackage{amsfonts}              
\usepackage{amsthm}                
\usepackage{color}
\usepackage{tabulary}
\usepackage{caption}
\usepackage{subcaption}
\usepackage{amssymb}
\usepackage{todonotes}
\theoremstyle{plain}
\usepackage{mathtools}

\usepackage{tikz}
\usetikzlibrary{calc,decorations.markings}

\newtheorem{thm}{Theorem}[section]
\newtheorem*{uthm}{Theorem}
\newtheorem{lem}[thm]{Lemma}

\newtheorem{prop}[thm]{Proposition}
\newtheorem{cor}[thm]{Corollary}

\newtheorem{rmk}[thm]{Remark}
\newtheorem{probl}[thm]{Problem}
\newtheorem{ques}[thm]{Question} 

\theoremstyle{definition}

\DeclareMathOperator{\LO}{LO}

\begin{document}

\title{Foliations from left orders} 
\author{Hyungryul Baik}
\address{%
		Department of Mathematical Sciences, KAIST\\
		291 Daehak-ro Yuseong-gu, Daejeon, 34141, South Korea 
}
\email{%
        hrbaik@kaist.ac.kr
}

\author{Sebastian Hensel}

\address{%
Mathematics Institute, University of Munich
Theresienstr. 39
D-80333 Munich
Germany		
}
\email{%
        hensel@math.lmu.de
        }

\author{Chenxi Wu}

\address{%
		Department of Mathematics
480 Lincoln Drive
213 Van Vleck Hall
Madison, WI 53706
}
\email{%
        cwu367@math.wisc.edu
}

\begin{abstract} We describe a construction which takes as an input a
  left order of the fundamental group of a manifold, and outputs a
  (singular) foliation of this manifold which is analogous to a taut foliation. 
  We investigate this construction in detail in dimension $2$, and exhibit connections to
  various problems in dimension $3$.
\end{abstract}

\maketitle

\section*{Introduction} 

A group is called left-orderable (or simply $\LO$) if it admits a
left-invariant linear order (a linear order $<$ on $G$ so that $a<b$
implies $ca< cb$ for all $a,b,c \in G$). We call such an order simply 
a \emph{left order}. This seemingly purely
algebraic notion is actually related to quite different aspects of the
group. We can only share a tiny part of the vast literature on
orderable groups here; compare e.g. the book of Mura and Rhemtulla
\cite{mura1978orderable} for a discussion of early results, or
\cite{clay2016ordered} and references therein for deep connections
between topology and orderability of groups has been discovered

For us, a crucial insight will be that orderablilty has a dynamical
interpretation: a countable group $G$ is $\LO$ if and
only if it admits a faithful action on the real line by
orientation-preserving homeomorphisms. To learn more about the
dynamical point of view of orderable groups, one can consult
\cite{navas2010dynamics} for instance.

For a closed orientable irreducible 3-manifold $M$, Zhao
\cite{zhao2021left} very recently presented a construction of
foliation $\mathcal{F}$ in $M \setminus int(B^3)$ where $B^3$ is a
3-dimensional ball embedded in $M$ with certain transverse
structure. $\mathcal{F}$ is analogous to a taut foliation in the sense
that there is a finite collection of arcs and simple closed curves
which transversely meet every left of $\mathcal{F}$. 

\smallskip 
In this paper, we
give an alternative construction of such foliations under the
assumption that $M$ has a strongly essential 1-vertex
triangulation. 
An one-vertex triangulation of a manifold is called
\emph{essential} if no edge-loop is null-homotopic, and \emph{strongly
  essential} if in addition no two edges are homotopic (fixing the
vertex). In low dimensions, the such triangulations are abundant; in
fact it has been conjectured that every closed orientable irreducible
3-manifold admits such a triangulation. A result of
Hodgson-Rubinstein-Segerman-Tillmann \cite{hodgson2015triangulations}
established this conjecture for a large class of 3-manifolds
(namely, either a closed Haken 3-manifold, a non-compact hyperbolic 3-manifold, or a closed 3-manifold with a Riemannian metric of constant negative or zero curvature).
 
Our main result is the following construction (which works in arbitrary dimension) and produces
singular foliations from orders.
\begin{uthm} Suppose $M$ admits an essential triangulation in $M$ with a unique vertex $v$ and $\pi_1(M)$ is $\LO$. Then to each left order $\mathcal{O}$ on $\pi_1(M)$ we associate a singular
	foliation $\mathcal{F}_\mathcal{O}$ which  has a unique singular point at $v$ and it is normal to each top-dimensional simplex of $\mathcal{T}$. 
\end{uthm} 

See Section~\ref{sec:singular-foliation} for details of the construction. Our foliation is also analogous to a taut foliation in the same sense as the one constructed by Zhao is. In fact, since the foliation is normal to each top-dimensional simplex of the triangulation, the set of edges of the triangulation is a collection of simple closed curves which meet transversely every leaf of the foliation. 

The singular foliation is called a \textbf{fLOiation}, suggesting that it is a foliation coming from a left order. 

We investigate this construction thoroughly in dimension 2. In particular, we give a complete classification of fLOiations in the $2$--torus in Section~\ref{sec:surfaces}. In the case of hyperbolic surfaces, such a classification seems much harder. However, we can prove the following first steps
\begin{uthm}
	Let $\mathcal{O}$ be a left order on $\pi_1(S_g)$ for $g \geq 2$. Then, each leaf of $\mathcal{F}_\mathcal{O}$ lifts to a quasigeodesic in $\mathbb{H}^2$. In particular, $\mathcal{F}_\mathcal{O}$  can be pulled tight to a geodesic lamination $\lambda_\mathcal{O}$. This assignment is continuous with respect to the topology on orders introduced by Sikora in \cite{sikora2004topology}, and the Hausdorff topology on the space of geodesic laminations.
\end{uthm}
Given this theorem, one can first try to characterise which geodesic laminations can be obtained as straightenings of
fLOiations. We give some positive and negative results in Section~\ref{sec:surfaces}: in particular, minimal orientable geodesic laminations are possible.

\smallskip
In Section~\ref{sec:dim3}, we study the case of dimension 3, and we give a necessary and sufficient condition for a fLOiation in a 3-manifold to be an honest regular foliation. The exact nature of fLOiations in dimension 3 is still very mysterious and we pose a lot of questions in this section. Some of the questions are related to so-called the L-space conjecture, proposed in the seminal paper of Boyer-Gordon-Watson \cite{boyer2013spaces}. It states that a rational homology 3-sphere $Y$ is not an L-space if and only if its fundamental group $G = \pi_1(Y)$ is $\LO$. Although the conjecture is still widely open, it suggests a deep connection between topology of a manifold and algebraic property of its fundamental group and has been studied by many authors, including \cite{boyer2005orderable}, \cite{dunfield2020floer}.  Later a connection to foliations is also made. Ozsv{\'a}th and Szab{\'o} \cite{ozsvath2004holomorphic} showed that a rational homology 3-sphere $Y$ has a taut foliation then it is not an L-space (see also \cite{kazez2017c0}, \cite{bowden2016approximating}), and the converse was conjectured by Juh\'asz in \cite{juhasz2015survey}. From this point of view, one can ask if in general the existence of a taut foliation in a manifold is equivalent for its fundamental group being $\LO$. We hope that further investigation on fLOiation in the future would shed some light in this direction too.

\subsection*{Acknowledgements} 
The construction of fLOiations was inspired by a conversation that the first author had with with Bill Thurston, Dylan Thurston, and Sergio Fenley in Spring 2012, and the name fLOiation was suggested by Bram Petri. We thank them greatly. 

We also thank Danny Calegari, Nathan Dunfield, Ursula Hamenst\"adt for helpful discussions. 
The first author was partially supported by Samsung Science \& Technology Foundation grant No. SSTF-BA1702-01.

\section{Two constructions of fLOiations}
\label{sec:singular-foliation}

In this section, we describe our constructions of singular foliations from orders. We begin with a general construction that works in any dimension. Then, we briefly describe an alternate, equivalent construction in dimension $2$, which is slightly easier to visualize. Finally, we discuss how the construction can be adapted to work with multi-vertex triangulations, and (certain) partial orders.

\subsection{General construction of fLOiations}
\label{subsec:fLOiationgeneral} 


Throughout, we let $M$ be a manifold, and $\mathcal{T}$ a fixed triangulation with a single vertex.
Under the assumption that $\pi_1(M)$ is left-orderable, there is a faithful action
\[ \rho:\pi_1(M) \to \mathrm{Homeo}(\mathbb{R}) \]
with the additional property that the orbit map of $0$ is injective (see Proposition 2.1 in \cite{navas2010dynamics} for instance). We call such $\rho$ \emph{allowed}.

We now describe a construction of a singular foliation from $\mathcal{T}$ and $\rho$ (and discuss dependence later).
Let $\tilde{M}$ be the universal cover of $M$, and lift the triangulation $\mathcal{T}$ from $M$ to a
triangulation $\tilde{\mathcal{T}}$ of $\tilde{M}$. We equip $M$ and $\tilde{M}$ with PL structures based on these triangulations. This makes the covering map PL, which in turn implies that the deck group $G= \pi_1(M)$ acts as PL transformations.
 
Denote the vertex of $\mathcal{T}$ by $v$, and let $v_0$ be one of its lifts. The preimage of $v$ under covering is $Gv_0$. We inductively define a map $F : \tilde{M}$ to $\mathbb{R}$ satisfying the following:
\begin{itemize}
	\item For each vertex $gv_0$, we have $F(gv_0) = \rho(g)(0)$.
	\item For edge $e$, pick an orientation such that the value of $F$ on the endpoint is greater than the starting point. Suppose the end point is $gv_0$. Now, if $g=e$, then we define $F$ to be linear on $e$. If $g\not=1$, then for any $x\in e$ we set $F(x)=\rho(g)(F(g^{-1}x))$.
	\item For $d$-dimensional cells, we extend $F$ from the edges to the cell in such a way that the level sets are  linear $d-1$ polytopes.
\end{itemize}

\begin{lem}
	The level sets of $F$ are invariant under the action of the deck group $\pi_1(M)$.
\end{lem}
\begin{proof}
	First, we observe that by definition, for any $x$ in the 1-skeleton of $\tilde{M}$, and any $g\in\pi_1(M)$, we have \[ gF(x)=F(gx). \] Hence, the deck group action preserves the level sets of $F$ restricted to the $1$--skeleton.
	
	Now, let $\Delta$ be any higher dimensional simplex. The intersection of a level set $F^{-1}(x)$ with $\Delta$
	is the polyhedron in $\Delta$ whose vertices are exactly those points in the $1$--skeleton $\Delta^1$ which lie
	in $F^{-1}(x)$. Since the deck group action acts as PL transformations, and preserves the level sets of $F$ on
	the $1$--skeleton, this implies that the deck group action preserves all level sets.
\end{proof}
Let $\tilde{\mathcal{F}}$ be the singular foliation defined by the level sets of $F$. By the previous lemma,
it descends to a singular foliation $\mathcal{F} = \mathcal{F}(\mathcal{T},\rho)$ on $M$.

The rest of this subsection is concerned with a discussion of the dependence on $\rho$. In particular, we would
like to obtain a construction which depends only on an order of $\pi_1(M)$, not an action on $\mathbb{R}$. The first
step is the following easy observation.
\begin{lem}\label{lem:obvious-change1}
	If $f:\mathbb{R}\to \mathbb{R}$ is any orientation-preserving homeomorphism which fixes $0$, then
	$\mathcal{F}(\mathcal{T},\rho) = \mathcal{F}(\mathcal{T},f\rho)$, where $f\rho$ denotes the representation
	where each element is postcomposed with $f$. 
\end{lem}
\begin{proof}
	Denote by $F$ and $F'$ the functions defining the covers of $\mathcal{F}(\mathcal{T},\rho), \mathcal{F}(\mathcal{T},f\rho)$. One then observes that 
	\[ F(x) = t \quad \Leftrightarrow \quad F'(x) = f(t). \]
	Namely, this is true by definition on the $1$--skeleton of the triangulation on $\tilde{M}$, and
	follows with an argument as in the previous lemma for all higher dimensional simplices. Thus, $F$ and $F'$ have the same level sets.
\end{proof}

In order to define the singular foliation from an order, we therefore aim to construct an (allowed) action
$\rho$ from the order, without choices. As a first step, we begin with an
order-preserving map 
\[ i: \pi_1(M) \rightarrow\mathbb{R} \] 
such that the identity $e$ is sent to $0$. We now promote $i$ to an allowed action
\[ \rho_i :\pi_1(M) \to \mathrm{Homeo}^+(\mathbb{R}) \]
in the following way:
\begin{enumerate}
	\item Define the action of $g\in \pi_1(M)$ on $i(\pi_1(M))$ by $g(i(h))=i(gh)$.
	\item Extend the action above to an action on $\mathbb{R}$ by homeomorphisms linearly in all complementary
	intervals $\mathbb{R}-i(\pi_1(M))$.
\end{enumerate}
Hence, we can now define a foliation $\mathcal{F}(\mathcal{T}, i) = \mathcal{F}(\mathcal{T}, \rho_i)$, and need
to discuss how the construction depends on $i$.
\begin{cor}\label{cor:obvious-change2}
	Any two (ambiently) isotopic embeddings $i, i'$ yield the same foliations $\mathcal{F}(\mathcal{T}, i) = \mathcal{F}(\mathcal{T}, i')$.
\end{cor}
\begin{proof}
	If $i, i'$ are ambiently isotopic, there is a homeomorphism $f:\mathbb{R}\to\mathbb{R}$ so that
	$i' = f \circ i$. In fact, we may assume that $f$ is affine on each complementary interval in
	$\mathbb{R}-i(\pi_1(M))$.
	By definition, we then have
	\[ \rho_{i'} = f\rho_i, \]
	and thus $\mathcal{F}(\mathcal{T}, i) = \mathcal{F}(\mathcal{T}, i')$ by Lemma~\ref{lem:obvious-change1}.
\end{proof}

As a last step, we promote the construction to start with the order itself, not the embedding $i$. To this end,
given any totally ordered countable set $S=\{s_0,s_1,\dots\}$, we define a ``minimal embedding'' $i_0$ of $S$ into $\mathbb{R}$ inductively as follows:
\begin{enumerate}
	\item $i_0(s_0)=0$.
	\item For any $n>0$, if $s_n>s_p, \forall p<n$, let $i_0(s_n)=\max_{p<n}i_0(s_p)+1$; if  $s_n<s_p, \forall p<n$, let $i_0(s_n)=\min_{p<n}i_0(s_p)-1$; if $s_a<s_n<s_b$ and $s_a$ and $s_b$ are the closest elements to $s_n$ in the finite sets $\{s_0, \dots s_{n-1}\}$, from below and from above respectively, then $i_0(s_n)={i_0(s_a)+i_0(s_b)\over 2}$.
\end{enumerate}
In fact, we can describe the difference between the foliations defined by minimal embedding and arbitrary ones.

\begin{lem} \label{lem:striplemma}
	Given any totally ordered countable set $S=\{s_0,s_1,\dots\}$, any other order preserving map from $S$ to $\mathbb{R}$ can be ``collapsed'' into $i_0$. 
	In other words, two different choices of $i$ differ only in the addition and removal of disjoint intervals, and
	the application of an ambient isotopy.
\end{lem}
\begin{proof}
	It is enough to construct a continuous, monotone (but generally not strictly monotone) map from $\mathbb{R}$ to itself that sends $i(G)$ to $i_0(G)$. This can be done in the following two steps:
	\begin{enumerate}
		\item On $i(G)$, let $f(i(g))=i_0(g)$.
		\item Now we extend it to $\overline{i(G)}$. Suppose $x\in \overline{i(G)}\backslash i(G)$, then either there is a sequence in $i(G)$ converging to $x$ from above but no from below, or there is a sequence from below but not above, or there are sequences converging to it from both above and below. In the previous two cases $f$ shall send $x$ to the 	limit of the image of those converging sequences under $f$. 
	\end{enumerate}
\end{proof}
\begin{cor}\label{cor:effect-of-i}
	If $i, i'$ are any two order-preserving embeddings of $\pi_1(M)$ into $\mathbb{R}$, then the foliations
	$\mathcal{F}(\mathcal{T}, i), \mathcal{F}(\mathcal{T}, i')$ differ by inserting or removing ``strips'' $L\times (0,1)$ of parallel leaves.
\end{cor}
\begin{proof}
	By Corollary~\ref{cor:obvious-change2} and Lemma~\ref{lem:striplemma} it suffices to analyse the effect
	of modifying $i$ by inserting or collapsing an interval (not containing any points from the embedding $i$).
	Now, from the construction it is immediate that this corresponds to blowing up a leaf into a band of parallel leaves (or collapsing such a strip). 
\end{proof}

In general, it is not clear to us how the choice of $\mathcal{T}$ influences the foliation. In low dimensions, where
different triangulations can easily be connected combinatorial, it seems likely that the choice of triangulation
is immaterial, but the general case is less clear.
\begin{ques}
	Describe the relation between $\mathcal{F}(\mathcal{T}, \mathcal{O})$ and $\mathcal{F}(\mathcal{T}', \mathcal{O})$ for different $1$--vertex triangulations.
\end{ques}

\subsection{Alternative construction of the fLOiations}

In this subsection, we describe an alternative, more explicit construction of fLOiations in dimension 2. 

Consider a triangle $\Delta$ with 3 edges $\alpha$, $\beta$ and $\gamma$ in the 2-skeleton of the quotient complex $\tilde{M}/\pi_1(M)$ above, oriented so that all three edges represents positive elements in $\pi_1(M)$, and labelled such that $\alpha$ followed by $\beta$ is homotopic to $\gamma$. For each of the three edges $e=\alpha$, $\beta$ or $\gamma$, we parameterize it with an interval $[0, -i(e^{-1})]$ as follows: firstly, because $\pi_1(M)$ acts freely on the preimages of $e$, we can identify $e$ with the preimage that ends with the point $v_0$. Now the identification between this edge and $[0, -i(e^{-1})]$ is done via the continuous map $F$ (defining the fLOiation, as above) composed with multiplication by $-1$.

Following leaves of $\mathcal{F}$ identifies $\beta$ with a subinterval $[0, -i(\beta^{-1})]$ of $\gamma$. In the parameterization of edges described above, this identification is simply the identity. In the same way, $\alpha$ is identified with another subinterval $[-i(\beta^{-1}), -i(\gamma^{-1})]$ of $\gamma$, via the map $x\mapsto -\beta^{-1}(-x)$ (in the parametrisations given above).

Hence, we could alternatively define $\mathcal{F}$ by choosing the parametrisations of the edges first, and joining them with line segments according to the rule in the previous paragraph.


\begin{figure}
  \begin{tikzpicture}[scale=1.5, ray/.style={decoration={markings,mark=at position .5 with {
          \arrow[>=latex]{>}}},postaction=decorate}]
    \draw[ray](0, 0)--(0, 2.5);
    \draw[ray](0, 0)--(1.5, 1.5);
    \draw[ray](1.5, 1.5)--(0, 2.5);
    
    \draw[thick, magenta, -](0, 0)--(0, 1.5);
    \draw[thick, green, -](0, 1.5)--(0, 2.5);
    \draw[thick, cyan, -](1.5, 1.5)--(0, 2.5);
    \draw[thick, blue, -](0, 0)--(1.5, 1.5);
    \draw[purple, -](0, 1.5)--(1.5, 1.5);
    \draw[orange, -](0, 2)--(0.75, 2);
    \draw[orange, -](0, 1.75)--(1.125, 1.75);
    \draw[orange, -](0, 2.25)--(0.375, 2.25);
    \draw[purple, -](0, 0.75)--(0.75, 0.75);
    \draw[purple, -](0, 0.5)--(0.5, 0.5);
    \draw[purple, -](0, 0.25)--(0.25, 0.25);
    \draw[purple, -](0, 1)--(1,1);
    \draw[purple, -](0, 1.25)--(1.25, 1.25);
    \node at (-0.2, 1.25) {$\gamma$};
    \node at (0.8, 0.6) {$\alpha$};
    \node at (0.8, 2.3) {$\beta$};
    \draw [orange, ->] (0.7, 1.8)--(1.5, 2.6);
    \node at (1.5, 2.8) {\color{orange}Identity\color{black}};
    \draw [purple, ->] (0.2, 0.4)--(0.8, -0.2);
    \node at (0.8, -0.4) {\color{purple}$\beta$\color{black}};
    \draw [-] (2, 0.5)--(6, 0.5);
    \draw [-] (2, 1.5)--(6, 1.5);
    \draw [-] (2, 2.5)--(6, 2.5);
    \draw [-] (2.5, 0)--(2.5, 3);
    \node at (2.3, 0.3) {$0$};
    \node at (2.3, 1.3) {$0$};
    \node at (2.3, 2.3) {$0$};
    \draw [-, very thick, blue](2.5, 1.5)--(5.5, 1.5);
    \draw [-, very thick, cyan](2.5, 0.51)--(3.5, 0.51);
    \draw [-, very thick, green](2.5, 0.48)--(3.5, 0.48);
    \draw [-, very thick, magenta](3.5, 0.5)--(5, 0.5);
    \node at (3.5, 0.3){\color{blue}$\beta(0)$\color{black}};
    \node at (5.3, 0.3){\color{blue}$\beta(\alpha(0))=\gamma(0)$\color{black}};
    \draw [->, blue](2.5, 1.5)--(3.5, 0.5);
    \draw [->, blue](5.5, 1.5)--(5, 0.5);
    \draw [->, purple](3.5, 1.5)--(4, 0.5);
    \draw [->, purple](4.5, 1.5)--(4.5, 0.5);
    \draw [->, red](2.5, 2.5)--(5.5, 1.5);
    \node at (5.7, 1.7){\color{red}$\alpha(0)$\color{black}};
    \node at (4.3, 1){\color{purple}$\beta$\color{black}};
    \node at (6.3, 0.5){$\mathbb{R}$};
    \node at (6.3, 1.5){$\mathbb{R}$};
    \node at (6.3, 2.5){$\mathbb{R}$};
    \draw [->, red] (6.3, 2.3)--(6.3, 1.7);
    \draw [->, blue] (6.3, 1.3)--(6.3, 0.7);
    \node at (6.5, 2){\color{red}$\alpha$\color{black}};
    \node at (6.5, 1){\color{blue}$\beta$\color{black}};
  \end{tikzpicture}
\caption{Foliation on a 2-simplex}
\end{figure}
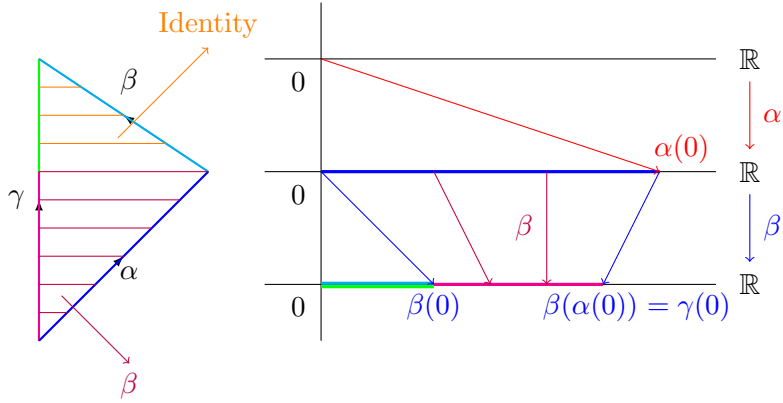


\subsection{Generalization to multi-vertex triangulations}

The above construction can be generalized to the case when the triangulation has more than one vertex: let $\tilde{M}$ be the universal cover, $D$ a fundamental domain of the deck group action, $V$ the set of vertices of $\tilde{M}$.
Choose a partial ordering on $V$ satisfying that $gv<v$ iff $g<e$. If $i$ is an order preserving embedding of $V$ into $\mathbb{R}$, then a faithful action of $G$ on $\mathbb{R}$ can be defined as above by $g(i(v))=i(gv)$ on $i(V)$ and extended linearly to $\mathbb{R}$. Now consider a map $F$ from $\tilde{M}$ to $\mathbb{R}$, defined as $i$ on the 0-skeleton, and extend to higher dimensional cells as follows:
\begin{itemize}
\item For 1-cell $e$, pick an orientation such that the value of $F$ on the end point is greater than the starting point. Now, if $e\in D$, then $F$ is linear on $e$. If $e=ge_0$, $e_0\in D$, then for any $x\in e$, $F(x)=gF(g^{-1}x)$.
\item For $d$-dimensional cells, we extend the function from the 1-faces to the cell in such a way that the level sets are all linear $d-1$ polytopes.
\end{itemize}

Arguing as in the one-vertex case, we now see that the level sets of $F$ are regular unless they pass through one of the vertices, and are invariant under deck transformations. Hence, it covers a singular foliation of $M$ as before.

\subsection{Generalization to partial ordering} \label{subsec:fLOiation_partial}

The construction for fLOiation can also be generalized to certain partial orders. More specifically, suppose that
we are given a partial order $<$ on $\pi_1(M)$ satisfying the property that ``$a<b$ and $b<a$ are both not true'' is an equivalence relation $\sim$. Instead of strongly essential, we now assume that the fundamental group elements corresponding to edges of $\mathcal{T}$ do not lie in the same equivalence class as the unit. 

Note that there is a total ordering on $G/\sim$ induced by this partial ordering, hence one can embed $G/\sim$ into $\mathbb{R}$, which provides a map from $G$ to $\mathbb{R}$, and then carry out the same construction as above. Since the group elements corresponding to edges are not equivalent to the unit under $\sim$, the result is again a singular foliation.

This construction will be useful in trying to understand fLOiation for orders which are constructed from
partial orders. For later use, we record the following:
\begin{rmk} \label{rem:blowup} Let $i$ be an embedding of $G/\sim$
  into $\mathbb{R}$, and let $J$ be a union of disjoint intervals of
  $\mathbb{R}$. Define an equivalence relation $\sim'$ by declaring
  $a\sim' b$ if and only if $i(a)$ and $i(b)$ lies in the same
  connected component of $J$. Let $j$ be the embedding of $G/\sim'$ to
  $\mathbb{R}$ obtained from $i$ by composing with a map that
  collapses each connected component of $J$.

  If $\sim'$ also satisfy the condition that fundamental group
  elements corresponding to edges do not lie in the same equivalence
  class, then the fLOiation defined $j$ is obtained from the fLOiation
  defined by $i$ by collapsing the preimage of each connected
  component of $J$ into a single leaf.
\end{rmk}
	
\section{Surfaces}\label{sec:surfaces}
In this section we discuss fLOiation on surfaces, where more precise questions can be answered.

\subsection{Torus Examples} 
Here we discuss examples of the construction on the two-dimensional torus $T$. 
We identify $\mathbb{R}^2$ with the universal cover of $S^1\times S^1$
in the usual way, and $\mathbb{Z}^2$ with the deck group as translations.

In this situation, there is a geometric way to construct orders on $\pi_1(T)$. Namely,
choose an (oriented) line $L \subset \mathbb{R}^2$ through the origin of $\mathbb{R}^2$, and denote by $V$ the halfspace to the left of $L$.  If
$L$ has irrational slope, then we declare all elements of $\pi_1(T)$ in $V$ to be positive. 
If $L$ has rational slope, we need to furthermore cut $V$ at $L^\perp$, and declare one of the halves to be positive.
Any order on $\mathbb{Z}^2$ is (up to isomorphism) one of these two classes (orders on $\mathbb{Z}^n$ are classified in \cite{teh1961construction}, \cite{robbiano1985term}. See Section~2 of \cite{koberda2011faithful} also).

\medskip If $L$ has irrational slope, then we let $\pi:\mathbb{R}^2\to L$ denote orthogonal projection.
 Then
under an identification of $L$ with $\mathbb{R}$, the restriction
$\pi\vert_{\mathbb{Z}^2}:\mathbb{Z}^2 \to L = \mathbb{R}$ yields a
bi-invariant order on $\mathbb{Z}^2$, via the action as translations.
Now, as the
defining function $F$ for the definition of the fLOiation we can simply
take the projection $\pi$ -- which implies that the resulting fLOiation is
the (irrational slope) foliation by lines orthogonal to $L$.

\medskip If $L$ has rational slope, then the orthogonal projection $\pi:\mathbb{R}^2\to L$ has a kernel $K \simeq \mathbb{Z}$. The choice of the halfspace in $V$ is equivalent to the choice of an order on $K$, and the order is
then the lexicographic order defined by these two orders. We can thus modify $\pi:\mathbb{Z}^2\to L$ to the desired function $F$ by ``blowing up'' the maps $a+K \to \pi(a)$ into injective maps to small intervals. In particular, observe that we can do this so that the resulting $F$ still has connected level sets (compare Remark~\ref{rem:blowup}).

Now, consider any element $k \in K$. Observe that then $k^n$ is bounded (in the order) by any element which maps under
$\pi$ to a positive number. This implies that the geodesic in $\mathbb{R}^2$ corresponding to $k$ converges to the leaf of the fLOiation corresponding to $\lim_{n\to\infty}F(k^n)$. Swapping perspective, this shows that 
the fLOiation (on $T$) has a leaf which spirals into the image of this geodesic (since level sets are connected). Hence, the fLOiation on $T$ has
a closed leaf in the homotopy class defined by $k$.

Summarizing, we have the following result connecting algebraic and topological properties.
\begin{prop}\label{prop:torus-case}
	Let $O$ be an order on $\pi_1(T) = \mathbb{Z}^2$, and $F_O$ the corresponding fLOiation. Then $F_O$ has a closed
	leaf if and only if $O$ is not Archimedean. 
\end{prop} 

\subsection{Geodesic Laminations on Hyperbolic Surfaces}
\label{sec:lamination} 

In this section, we focus on closed hyperbolic surfaces $\Sigma$ of genus $g \geq 2$, 
and $\Delta$ be a one-vertex triangulation of $\Sigma$, all edges of which are essential simple
closed curves.  Consider the universal cover $\mathbb{H}^2$ of
$\Sigma$, and the preimage $\widetilde{\Delta}$ of $\Delta$ in
$\mathbb{H}^2$. Let $\gamma$ be a side of a triangle in $\Delta$. The
preimage of $\gamma$ in $\mathbb{H}^2$ consists of countably many embedded infinite 
paths.
\begin{lem} There is a constant $K>0$ so that 
  each component $\widetilde{\gamma}$ of the preimage of an edge $\gamma$ is a $K$--quasigeodesic. 
\end{lem}
\begin{proof}
  The vertices of $\widetilde{\Delta}$ are labelled by elements of the
  deck group $\pi_1(\Sigma)$. A component $\widetilde{\gamma}$ as in the lemma contains vertices of the form 
  $c\cdot g^i, i \in \mathbb{Z}$ for some $c \in \pi_1(\Sigma)$, and $g$ the label of the edge $\gamma$. 
  Now recall that $\pi_1(\Sigma)$ is a hyperbolic group, and any infinite cyclic subgroup of a hyperbolic group is 
  undistorted \cite{bridson2013metric}. Thus, $c\cdot g^i, i \in \mathbb{Z}$ is a quasigeodesic in $\pi_1(\Sigma)$, and therefore $\mathbb{H}^2$. Hence, so is $\widetilde{\gamma}$ (as the segments connecting consecutive vertices are lifts of
    $\gamma$ and therefore have the same, finite length). The uniformity statement follows as the triangulation only has finitely many edges.
\end{proof}
We call a preimage $\widetilde{\gamma}$ of an edge $\gamma$  as in the lemma a
\emph{wall (of type $\gamma$)}. The complement of $\widetilde{\gamma}$
in $\mathbb{H}^2$ thus consists of two (quasiconvex) halfspaces. We denote
their boundary intervals by $I(\widetilde{\gamma})_i, i=1,2$. They intersect in the boundary
points of  $\widetilde{\gamma}$.

\smallskip
The first main result of this section is the following
\begin{prop}\label{prop:endpoint-criteria}
  Suppose that $L:[0, \infty) \to \mathbb{H}^2$ is an embedded path
  so that each intersection with a wall is transverse and in at most one point, and which intersects
  infinitely many walls. Then $L$ has a well-defined endpoint in the
  circle at infinity.

  If $L:(-\infty, \infty) \to \mathbb{H}^2$ intersects each wall it intersects transversely, and in at most one point, and
  both the positive and negative ray intersect infinitely many walls, then the corresponding
  endpoints are distinct.
\end{prop}
\begin{proof}
  Let $0\leq t_1 < t_2 < \cdots$ be all positive real numbers so that
  $L(t_i)$ lies in a wall $w_i$. By the assumption that $L$ intersects
  each wall in at most one point, the path $L:[t_i, \infty) \to
  \mathbb{H}^2$ is contained in one of the halfspaces $C_i$ defined by
  $w_i$; we let $I_i$ be the boundary interval of that halfspace.  To
  show the proposition, it suffices to show that $\bigcap_i I_i$
  consists of a single point.  Namely, since the $w_i$ are
  quasigeodesics with uniform constants, this would imply that the
  intersection of the closures $\overline{C_i}$ of all $C_i$ also
  consists of a single point at infinity.

  \smallskip Since $\Delta$ has only finitely many edges, there are
  indices $n_1 < n_2 < \cdots$ so that all $w_{n_i}$ are lifts
  $\widetilde{\gamma}_i$ of the same edge $\gamma$ of $\Delta$. The
  corresponding halfspaces $C_{n_i}$ are nested (since the
  $\widetilde{\gamma}_i$ are all disjoint). Assuming that $\bigcap_i I_i$ is 
  not a single point, we conclude that $\bigcap_i I_{n_i}$ is also not a single point.

  Thus, $\bigcap_i I_{n_i} = [\alpha, \beta]$ for some circular arc
  $[\alpha, \beta]\subset S^1$, and the endpoints of $\widetilde{\gamma}_i$
  converge to $\alpha, \beta$.

  Suppose that $\gamma$ corresponds to the element $g\in\pi_1(\Sigma)$
  in the deck group. Then there is a lift $\widetilde{\gamma}$ whose endpoints
  agree with the endpoints of the axis of $g$. Since all  $\widetilde{\gamma}_i$ are lifts
  of $\gamma$, we have that  $\widetilde{\gamma}_i = g_i \widetilde{\gamma}$.

  Thus, we conclude that the endpoints of the axes of the conjugates $g_igg_i^{-1}$
  (which agree with the endpoints of $\widetilde{\gamma}_i$) converge to $\alpha \neq \beta$.
  But this is impossible for a cocompact Fuchsian group.

  \medskip It remains to show the last claim about distinct endpoints. To this end, note that
  under the assumptions there will be numbers $r < 0 < s$ so that $L(r), L(s)$ are both on 
  walls which are lifts of the same edge $\gamma$. This implies that the corresponding walls
  are disjoint, and in fact the same is true for the boundary intervals of the halfspaces which
  $L(t), t\geq s$ and $L(t), t \leq r$ are contained in. This shows that the endpoints are distinct.
\end{proof}

We now aim to improve this argument to study $2$--dimensional fLOiations, and their dependence on
the order. The first is the following consequence of Proposition~\ref{prop:endpoint-criteria}
\begin{cor}\label{cor:essential-leaves}
	Let $\mathcal{F}$ be a fLOiation of $\Sigma$ determined by the $1$--vertex triangulation $\Delta$
	and an embedding $\rho: \pi_1(\Sigma,v) \to \mathrm{Homeo}^+(\mathbb{R})$. Then each leaf of $\mathcal{F}$
	lifts in $\widetilde{\Sigma}$ to a bi-infinite path with well-defined endpoints at infinity.
\end{cor}
\begin{proof}
	Given Proposition~\ref{prop:endpoint-criteria} we need to check that walls are intersected transversely,
	in at most one point, and infinitely often in both directions. Transversality of intersections is clear from the
	construction. 

	Next, we argue that no wall is intersected more than once by a leaf. Namely, suppose that a leaf segment $l_0$ would intersect the same wall $w$ in two points. This would
	imply that along the wall $w$, two points map to the same point in $\mathbb{R}$ under $F$ (as leaves are level sets). On the other hand, the wall consists of images $c g^i(e), i \in \mathbb{Z}$ for some edge $e$ and group
	elements $c, g$ (where $c$ corresponds to some point on the wall, and $g$ is the group element corresponding to the edge $e$). By construction of the function $F$ defining the fLOiation, it restricts
	on $c g^i(e)$ as an injective function taking values in $[\rho(cg^i), \rho(cg^{i+1})]$, where $\rho$ is the embedding of $G$ into $\mathbb{R}$. Hence, the restriction of $F$ to the wall is injective, showing the claim.
	
	Finally, observe that since every wall is intersected at most once, and every leaf intersects infinitely many edges of the triangulation in both directions, leaves also intersect infinitely many walls in both directions.
\end{proof}
As a consequence, \emph{every surface fLOiation has a well-defined straightening}. We denote
by
\[ S: \{\rho:\pi_1(\Sigma) \to \mathrm{Homeo}^+(\mathbb{R}) \mbox{ allowed}\} \to \mathcal{GL}(\Sigma) \]
the map which assigns to an allowed representation $\rho$ the straightening of the associated
fLOiation, seen as a geodesic lamination on $\Sigma$. Observe also that by results in Section~\ref{sec:singular-foliation}, the image
$S(\rho)$ only depends on the order defined by $\rho$, not $\rho$ itself.

Note that the map $S$ is in fact the composition of the maps 
$$\{\rho:\pi_1(\Sigma) \to \mathrm{Homeo}^+(\mathbb{R})\mbox{ allowed}\} \to \{ \mbox{fLOiations} \} \to \mathcal{GL}(\Sigma),$$
where the first map is our construction in Section 1 and the second map is the straightening which is allowed by Corollary \ref{cor:essential-leaves}. Abusing notation, we also just use $S$ to denote the second map in the composition. 

Next, we will discuss the dependence of a fLOiation on the embedding $\rho: \pi_1(\Sigma,v) \to \mathrm{Homeo}^+(\mathbb{R})$. The main ingredient is the following version of Proposition~\ref{prop:endpoint-criteria}.
\begin{prop}\label{prop:endpoint-convergence}
	Let $D\subset \mathbb{H}^2$ be a fundamental domain for the $\pi_1(\Sigma)$--action.
	For each $\epsilon>0$ there is a compact set $K \subset \mathbb{H}^2$ so that the following holds.
	
	Suppose that $L_1, L_2:[0, \infty) \to \mathbb{H}^2$ are embedded paths which intersect $D$, and which satisfy the prerequisites of Proposition~\ref{prop:endpoint-criteria}.
	Further suppose that for each wall $w$ which intersects $K$, the path $L_1$ intersects $w$ if
	and only if the path $L_2$ intersects $w$. 
	
	Then the endpoints of $L_1, L_2$ guaranteed by  Proposition~\ref{prop:endpoint-criteria} have distance
	at most $\epsilon$ (with respect to the standard angular metric on $S^1$).
\end{prop}
\begin{proof}
	By cocompactness of the action of $\pi_1(\Sigma)$ on $\mathbb{H}^2$, there are walls $w_1, \ldots, w_N$ so that the endpoints of each $w_i$ cut off an interval $I_i$ of 
	size $<\epsilon$ in $S^1$, and each point in $S^1$ lies in at least one such. Now, choose a compact set $K$ which intersects each $w_i$. By possibly choosing the $w_i$ further from the fundamental domain $D$, we may assume that
	there is no quasigeodesic with the constants from Proposition~\ref{prop:endpoint-criteria} which can intersect $D$
	and have both endpoints in $I_i$. 
	
	Suppose that $L_1, L_2$ satisfy the conditions of the proposition. Then, each $L_i$ is, by Proposition~\ref{prop:endpoint-criteria} and the choices above, a quasigeodesic whose endpoints lie in 
	lie in different of these intervals. Hence, $L_i$ has an endpoint in $I_j$ if and only if it intersects
	$w_j$. Now, by assumption, this means that $L_1, L_2$ have endpoints in the same intervals $I_r, I_s$. This shows the proposition.
\end{proof}

\begin{cor}\label{cor:continuity}
	The map $S$ is continuous, for the topologies of pointwise convergence on \newline $Hom(\pi_1(\Sigma, v), \mathrm{Homeo}^+(\mathbb{R}))$ and the Hausdorff topology on geodesic lamination space.
\end{cor}
\begin{proof}
	Let $D$ be a fundamental domain for the action of $\pi_1(\Sigma, v)$ on $\mathbb{H}^2$, and let $\rho_0$ be
	some allowed representation.
	
	Given any allowed representation $\rho$, denote by $F_\rho:\mathbb{H}^2\to \mathbb{R}$ the function
	defining the fLOiation. Observe that, by construction, for any compact set $K \subset \mathbb{H}^2$,
	and any $\epsilon>0$, if $\rho$ is close enough to $\rho_0$, then $F_\rho, F_{\rho_0}$ are $\epsilon$--close
	on $K$. In particular, intersections of level sets with $K$ can be guaranteed to be arbitrarily close to each other in the Hausdorff topology.
	
	This implies that, if $x \in D$ is a point so that the leaf $L_x$ of $\tilde{\mathcal{F}}_{\rho_0}$ through $x$ is nonsingular, and $w$ is any given wall $w$, then for any representation $\rho$ sufficiently close to $\rho_0$,
	the leaf of $\tilde{\mathcal{F}}_{\rho}$ through $x$ is also nonsingular, and intersects $w$.
	
	Hence as $\rho \to \rho_0$, the straightened fLOiations $S(\tilde{\mathcal{F}}_{\rho})$ have leaves converging to
	$S(L_x)$. Since we can take a finite number of leaves $L_{x_1},\ldots, L_{x_n}$ whose closure (on $\Sigma$)
	is the whole fLOiation, this implies that the straightened fLOiations $S(\tilde{\mathcal{F}}_{\rho})$ converge to a lamination which contains $S(\tilde{\mathcal{F}}_{\rho_0})$.
	
	Finally, suppose that this containment is proper. Then there is sequence of straightened leaves $S(L_i)$ of
	$\tilde{\mathcal{F}}_{\rho_i}$ which converge to a geodesic $g$ which is not contained in $S(\tilde{\mathcal{F}}_{\rho_0})$. Since laminations are closed, this implies that there is a pair of walls 
	$w_1, w_2$ intersected by $g$ so that no leaf of $S(\tilde{\mathcal{F}}_{\rho_0})$ intersects both.
	Up to passing to a subsequence we may assume that $L_i$ are leaves through
	points $y_i \in D$ which converge to $y\in D$. But now, if the leaf $L_y$ of $\tilde{\mathcal{F}}_{\rho_0}$ does not intersect $w_1$ or $w_2$, then the same will be true for the leaves $L_i$ by the above, which is a contradiction.
\end{proof}

\begin{cor} 
  For the canonical embedding and extension to allowed actions
  described in Section~\ref{sec:singular-foliation}, the assignment of the straightened
  fLOiation $S(\mathcal{F}(\rho_\mathcal{O}))$ to an order $\mathcal{O}$ on $\pi_1(\Sigma)$ is
  continuous. 
\end{cor}

       \subsection{The Image of the Straightening Map}

       We end this section with an explicit construction of laminations on
       hyperbolic surfaces obtained by straighening fLOiations.
       
       Let $\alpha$ be the real part of a holomorphic $1$-form
       $\omega$ on a closed Riemann surface $\Sigma$ with one single
       zero, so that any non zero element of $H_1(\Sigma)$ paired with
       $\alpha$ is non-zero. We further assume that there is no closed
       singular leaf in the foliation induced by $\alpha$.

       Recall that $\omega$ induces a piecewise Euclidean structure on
       the surface, with the only singular point being the unique
       zero. Choose a geodesic triangulation $\mathcal{T}$ of $\Sigma$ for this
       piecewise Euclidean structure, so that the zero is the unique
       vertex (such a triangulation always exist, for example, one can
       always choose the Delaunay triangulation), and as always lift the
       triangulation to a triangulation $\tilde{T}$ on
       $\tilde{\Sigma}$ via the covering map.

      Let $\pi_1(\Sigma)=G_0>G_1>G_2\dots$ be the lower
      central series, then there is an embedding $i_0$ from
      $H_1(\Sigma)=G_0/G_1$ to $\mathbb{R}$ which send the order
      induced by $\alpha$ to the order on $\mathbb{R}$.

      Now we try to build an embedding of $\pi_1(\Sigma)$ into $\mathbb{R}$ so that the induced order on $\pi_1(\Sigma)$ is left invariant:
      \begin{prop}
       Let $S$ be a closed hyperbolic surface. Then, given any left order on $H_1(\Sigma)=G_0/G_1$, there exists a left order on $G_0=\pi_1(\Sigma)$ extending the partial ordering induced by the given ordering on $G_0/G_1$ (which we call a {\em lexicographic ordering} on $G_0$). Furthermore, suppose $i_0: H_1(\Sigma)\rightarrow \mathbb{R}$ is an embedding that sends the given left order on $H_1(\Sigma)$ to the order in $\mathbb{R}$, then there is a embedding $i: G_0\rightarrow\mathbb{R}$, such that $i_0$ can be obtained by collapsing intervals in the image of $i$. (We call $i$ a {\em blow-up} of $i_0$.)
      \end{prop}

      \begin{proof}
        Firstly, for each $k$, fix a left order of $G_k/G_{k+1}$. Then, start with the embedding $i_0$ above. Replace each image of $i_0$ with an interval, such that the total length of such intervals is no more than $2^{-1}$, and send $G_1/G_2$ to the interval in a way that is order preserving, which defines a map $i_1: G/G_2\rightarrow \mathbb{R}$. Repeat the process to get $i_2$, $i_3$, $\dots$, where the new intervals added at the $j$-th step has total length no more than $2^{-j}$. Then $i$ is the limit of $i_j$ as $j\rightarrow\infty$.
	\end{proof}
	We now have maps $i_0, i:G  \to \mathbb{R}$. The map $i_0$ is not an embedding, but it satisfies the
	conditions discussed in Section~\ref{subsec:fLOiation_partial}, and so defines a fLOiation.
      \begin{prop}
	The foliation defined by $i_0$ is identical to the foliation on $\Sigma$ induced by $\alpha$.
\end{prop}
\begin{proof}
  By construction, $\omega$ as well as $\alpha$ are constant forms on each
  triangle on $\tilde{T}$. Hence, the integral of $\alpha$ is linear
  on each triangle, in particular they would be linear on edges
  starting at a given lifting of the vertex $v_0$. This implies that
  the integral of $\alpha$ is exactly the function $F$ in
  Section~\ref{sec:singular-foliation}, hence its level set is this
  new foliation.
\end{proof}

	Now, by Remark~\ref{rem:blowup}, the foliation defined by $i$ differs from the one defined by $i_0$ by blowing up certain leaves into strips.
    Due to the assumption on $\alpha$, each connected component of the preimage of any inserted interval during the blowup can contain only one vertex in $\tilde{\Sigma}$. Hence, the foliation on this connected component can only be isotopic to the foliation on a tubular neighborhood of the singular leaf of $Re(z^kdz)$ (which we shall call a ``star-shaped ribbon graph''). 

    Hence, the foliation formed from $i$ can be obtained by ``blowing up'' the foliation corresponding to $\alpha$, by replacing the singular leave with a star-shaped ribbon graph. This ribbon graph will become a $n$-gon after straightening just like the singular leaf of the foliation corresponding to $\alpha$.
    
    In particular, the straightening of $\mathcal{F}_i$ agrees with the straightening of the vertical foliation of $\alpha$.
     
    In conclusion, we have the following:
    \begin{prop}
      Suppose $\alpha$ is an Abelian differential with a single zero on $\Sigma$, so that
      the foliation defined by the real part of $\alpha$ has no closed leaves. Then the
      straightening of this foliation lies in the image of $S$.
    \end{prop}
    
Above proposition naturally lead us to the following question. 
\begin{ques}
	Characterise (in terms of properties of the order $O$ on $\pi_1(\Sigma_g)$) when the fLOiation $F_O$ has a closed leaf.
\end{ques}       
    
    We also remark the following, which follows since the exact choice of how to
    extend the partial order $i_0$ to $i$ is immaterial:
    \begin{cor}
      The map $S$ is not injective.
    \end{cor}

\section{Dimension 3}
\label{sec:dim3}

In this section, we focus in the case of dimension 3. First, we note that a construction similar to the one given in Section \ref{sec:surfaces} for surfaces can be done to the mapping torus $M$ of surface maps. This can be done via the central series: $G=\pi_1(M)>\pi_1(\Sigma)>\dots$, where $\Sigma$ is the surface, and $\pi_1(\Sigma)>\dots$ is the lower central series of $\pi_1(\Sigma)$. The one-vertex triangulation is formed from a one vertex triangulation of $S$, and the restriction of the resulting foliation on $\Sigma$ is as constructed above.

\medskip Next we ask when the fLOiation $\mathcal{F}$ constructed in Section~\ref{sec:singular-foliation} is nonsingular. 

In \cite{calegari2000foliations}, Calegari showed that orientations on the edges of a triangulation which satisfy certain conditions gives information about the 3-manifold. Let $M, \mathcal{T}, \mathcal{F}$ be as in Section~\ref{sec:singular-foliation}. We have orientations on the edges of $\mathcal{T}$ given by a left order on $\pi_1(M)$. In the section, we study when this choice of orientations satisfies the conditions that Calegari considered. 

We begin by recalling some terms introduced in \cite{calegari2000foliations}. 
Let $M$ be a closed 3-manifold, and let $\mathcal{T}$ be a triangulation of $M$. For each vertex $v$, we define the maximal subgraphs $o(v)$ and $i(v)$ of $link(v)$ 
whose vertices are, respectively, the outgoing and the incoming vertices from and to $v$. 
A \emph{direction} on $M$ is a choice of orientation for each edge in the 1-skeleton $\mathcal{T}$. A direction is a \emph{local orientation} if it satisfies the conditions
\begin{itemize}
\item[1.] for each vertex $v$ the graphs $o(v)$ and $i(v)$ are nonempty and connected
\item[2.] the direction restricts to a total ordering on the vertices of each tetrahedron
\item[3.] the 1-skeleton is recurrent as a directed graph. That is, there is an increasing path from each vertex to each other vertex.
\end{itemize} 

Suppose $M, \mathcal{T}$ are as in the previous section. The direction is given so that each edge represents a positive element of $\pi_1(M)$ with respect to the linear order to begin with. Since there is a total order on the entire group $\pi_1(M)$ which restricts to the direction on the edges of $\mathcal{T}$, the condition 2 in the definition of a local orientation is automatically satisfied. Condition 3 is also vacuously true, since there is only one vertex. 

To have a local orientation, now it suffices to check that condition 1 is also satisfied. What we show is the following. 

\begin{prop}
\label{prop:regular-criterion}
 The singular foliation $\mathcal{F}$ is a regular foliation if and only if the choice of orientations on the edges of $\mathcal{T}$ given by the left order is a local orientation. 
\end{prop} 
\begin{proof}

To show this, we consider the link of the vertex $v$. Since $M$ is a manifold and $\mathcal{T}$ is a triangulation, the $link(v)$ is actually a triangulated sphere. Let $\tau$ denote the triangulation of $link(v)$. We color the edges of $\tau$ in the following way: 
\begin{itemize}
\item one colors an edge $e$ in blue if the vertices connected by $e$ represents the local edges have the same orientation. 
\item otherwise, one colors $e$ in black.
\end{itemize}  
It is straightforward to see that this coloring is well-defined. Each tetrahedron in $\mathcal{T}$ gives four triangles of $\tau$, and obviously two of them have all edges colored in blue, and each of other two has one blue edge and two black edges. Let $n$ be the number of tetrahedra in $\mathcal{T}$. Then, there are $4n$ blue edges and $2n$ black edges in $\tau$. 

One each tetrahedron  $\Upsilon$ of $\mathcal{T}$, among connected components of $\mathcal{F} \cap \Upsilon$, there are two distinguished triangles which meet the vertex at the corner which are not local sink nor local source. Some faces of $\tau$ meets these distinguished triangles, and we record the intersection as a mid-curve colored in red. On the 2-skeleton of $\tau$, these form a union of closed red curves so that each face interests at most one red curve. 

Note that from the construction it is clear that both $o(v)$ and $i(v)$ are nonempty. 
First we show that the red locus has only one connected component if and only if $o(v)$ and $i(v)$ are connected. Observe that the complement of the interior of the carrier of the red curves (union of the triangles intersecting a red curve) has only blue edges. There are more than one connected components of the red locus if and only this complement has at least three connected components. Each connected component is entirely contained in either $o(v)$ or $i(v)$. Hence this means either $o(v)$ or $i(v)$ is not connected. 
s
To finish the proof, we need to show that the red locus has only one connected component if and only if the foliation $\mathcal{F}$ is non-singular. Note that each leaf has a natural cell-decomposition where each cell is a connected component of the intersection between the leaf and a tetrahedron. Hence each cell is either a triangle or a quadrilateral. Let $L$ be the leaf containing the vertex $v$. Note that all the cells of $L$ containing $v$ as a vertex are triangles which has a piece of red curve.  These are distinguished triangles in tetrahedra described above. Let $U$ be the union of all these triangles incident to $v$. Then components of $U \setminus v$ precisely corresponds to the components of the red locus. Now simply note that $\mathcal{F}$ is non-singular if and only if $U \setminus v$ has only one component. Now the proposition is proved. 

\end{proof}

Using the above proposition and the result of \cite{calegari2000foliations}, we can conclude that when $\mathcal{F}$ is a regular foliation (equivalently, the order-induced edge orientation on $\mathcal{T}$ is a local orientation), then the universal cover of $M$ admits a measured foliation. 

\begin{lem} The direction on the 1-skeleton of $\mathcal{T}$ given by the linear order has the property that every directed loop is homotopically essential. 
\end{lem} 
\begin{proof} This is a direct consequence of the fact that the positive cone for a linear order is a semi-group, i.e., a product of positive elements is again a positive element. 
\end{proof}

\begin{thm} If the singular foliation $\mathcal{F}$ is a regular foliation, then $\widetilde{M}$ admits a measured foliation which is normal with respect to the lift of $\mathcal{T}$. 
\end{thm} 
\begin{proof} This now directly follows from Theorem 5.1 of \cite{calegari2000foliations}. 
\end{proof} 

From our construction, a left order on $\pi_1(M)$ gives both an edge orientation and a singular foliation. Sometime an edge orientation is good enough to get a foliation. For example, Dunfield introduced a notion of foliar orientation in \cite{dunfield2020floer} and it is a sufficient condition for the manifold to admit a taut foliation. 
In the case a given order on $\pi_1(M)$ is bi-invariant, ie., both left and right invariant, the induced edge orientation necessarily admits a sink edge so cannot be foliar. Hence it is not a priori clear if the manifold admits a taut foliation or not. On the other hand, we still have a fLOiation in this case, and there is no obvious obstruction for the edge orientation to be a local orientation. Once it is a local orientation, by Proposition \ref{prop:regular-criterion}, the fLOiation is a regular foliation and actually it is taut since the collection of edges is a collection of simple closed curves which meet every leaf of the foliation transversely. Hence, this could provide a larger class of examples of edge orientations than foliar orientations which guarantee the existence of a taut foliation in the manifold (hence related to the L-space conjecture). In short, we propose the following question. 

\begin{ques}\label{ques:biorder} 
When is the edge orientation induced by a bi-invariant order a local orientation? 
\end{ques} 

Above question might be easier to answer when we have more information about bi-invariant orders of $\pi_1(M)$. 

\begin{probl} When the manifold is either seifert-fibered or Sol, then there is a known necessary and sufficient condition for $\pi_1$ to be bi-orderable by Boyer-Rolfsen-Wiest. 
Answer Question \ref{ques:biorder} in this case. 
\end{probl}

\bibliographystyle{abbrv}
\bibliography{biblio}

\end{document}